\documentclass[12pt]{article}

\usepackage{amsmath,amsthm,amsfonts,amssymb,enumerate,mathrsfs,mathtools,tikz}

\usetikzlibrary{matrix,arrows}

\newtheorem{theorem}{Theorem}[section]
\newtheorem{prop}[theorem]{Proposition}
\newtheorem{lemma}[theorem]{Lemma}
\newtheorem{cor}[theorem]{Corollary}
\theoremstyle{definition}
\newtheorem{defn}[theorem]{Definition}

\theoremstyle{remark}
\newtheorem{rmk}[theorem]{Remark}

\newtheorem*{theoremA*}{\bf Theorem A}
\newtheorem*{theoremB*}{\bf Theorem B}
\newtheorem{ex}[theorem]{Example}
\newtheorem{fact}[theorem]{\bf Fact}

\begin{document}

\newcommand{\val}{\mathcal{O}}
\newcommand{\can}{\mathcal{O}_c}
\newcommand{\M}{\mathcal{M}}
\newcommand{\Nod}{\mathcal{N}}
\newcommand{\fun}{\mathcal{F}}
\newcommand{\gun}{\mathcal{G}}
\newcommand{\lan}{\mathcal{L}}
\newcommand{\cat}{\mathcal{C}}
\newcommand{\pideal}{\mathfrak{p}}
\newcommand{\qideal}{\mathfrak{q}}
\newcommand{\sch}{\mathcal{L}_{sch}}
\newcommand{\alg}{\mathcal{L}_{alg}}
\newcommand{\Xn}{X^{\otimes n}}
\newcommand{\X}{X^f}
\newcommand{\Xm}{\mathcal{X}}
\newcommand{\Z}{\mathbb{Z}}
\newcommand\blfootnote[1]{%
  \begingroup
  \renewcommand\thefootnote{}\footnote{#1}%
  \addtocounter{footnote}{-1}%
  \endgroup
}

\title{Canonical Valuations and the Birational Section Conjecture}
\author{K. J. Str{\o}mmen}

\maketitle

\begin{abstract}
We develop a notion of a `canonical $\cat$-henselian valuation' for a class $\cat$ of field extensions, generalizing the construction of the canonical henselian valuation of a field. We use this to show that the $p$-adic valuation on a finite extension $F$ of $\mathbb{Q}_p$ can be recovered entirely (or up to some indeterminacy of the residue field) from various small quotients of $G_F$, the absolute Galois group of $F$. In particular, it can be recovered fully from the maximal solvable quotient. We use this to prove several versions of the birational section conjecture for varieties over $p$-adic fields.
\end{abstract}


\section{Introduction}

Let $X/K$ be a complete, smooth and geometrically irreducible curve over a field $K$, with function field $F:=K(X)$. Let $\hat{F}$ be any Galois extension, and put $\hat{K}:=K \cap \hat{F}$. Then the canonical projection map of Galois groups $pr:Gal(\hat{F}/F) \rightarrow Gal(\hat{K}/K)$ sits in an exact sequence
\begin{equation}
1 \rightarrow Gal(\hat{F}/F\hat{K}) \rightarrow Gal(\hat{F}/F) \rightarrow Gal(\hat{K}/K) \rightarrow 1 
\end{equation}
Given any $a \in X(K)$, we can assign to it a `bouquet' of group-theoretic sections $s_a: Gal(\hat{K}/K) \rightarrow Gal(\hat{F}/F)$. Indeed, let $v_a$ be the valuation on $F$ corresponding to $a$, and $w$ the valuation on $F\hat{K}$ corresponding to a preimage of $a$ in $\hat{X}:=X \otimes_K \hat{K}$ (so $w$ extends $v$). If we let $I_w$ and $D_w$ denote the inertia and decomposition group of $w/v$ inside $Gal(\hat{F}/F)$, then we get by Hilbert Decomposition Theory a commutative diagram
\begin{center}
\begin{tikzpicture}[>=angle 90]
\matrix(d)[matrix of math nodes, row sep=3em, column sep=2em, text height=1.5ex, text depth=0.25ex]
{1 & Gal(\hat{F}/F\hat{K}) & Gal(\hat{F}/F) & Gal(\hat{K}/K) & 1 \\
1 & I_w & D_w & G_w & 1\\};
\path[->, font=\scriptsize]
 (d-1-1) edge node[above]{} (d-1-2)
 (d-1-2) edge node{} (d-1-3)
 (d-1-3) edge node[above]{$pr$} (d-1-4)
 (d-1-4) edge node{} (d-1-5)
 (d-2-1) edge node[above]{} (d-2-2)
 (d-2-2) edge node{} (d-2-3)
 (d-2-3) edge node[above]{} (d-2-4)
 (d-2-4) edge node{} (d-2-5)
 (d-2-2) edge node{} (d-1-2)
 (d-2-3) edge node{} (d-1-3)
 (d-2-4) edge node[right]{$\simeq$} (d-1-4)
 ;
\end{tikzpicture}
\end{center}
with exact rows. Here $G_w$ denotes the Galois group of the residue field extension. It is known that the bottom row admits sections (see e.g. \cite{kpr}). Any choice of such induces a section $s_w$ of (4.1) such that $s(Gal(\hat{K}/K)) \subset D_w$, which is unique up to conjugation by an element of $Gal(\hat{F}/F\hat{K})$. Any member of the `bouquet' of sections obtained in this manner is said to lie over $a$. In a similar manner, if $v$ is a valuation which is trivial on $K$ and has residue field $K$, the same discussion shows that $v$ induces a `bouquet' of sections which are said to lie over $v$. We call such valuations {\bf $K$-valuations}

If we let $\mathcal{S}_{\hat{F}}$ denote the set of sections of (1) modulo conjugation, we have thus defined a map
\begin{equation}
\Psi_{\hat{F}}: X(K) \rightarrow \mathcal{S}_{\hat{F}}. 
\end{equation}

In particular, taking $\hat{F}=\overline{K(X)}$, this gives a map from $X(K)$ to sections of the exact sequence
\begin{equation}
1 \rightarrow G_{\overline{K}(X)} \rightarrow G_{K(X)} \rightarrow G_K \rightarrow 1,
\end{equation}
where for any field $F$ we let $G_F$ denote its absolute Galois group. As part of his visionary programme of `anabelian geometry', outlined in his famous ``Esquisse d'un Programme'' (see the appendix of \cite{schneps}), Grothendieck made the following conjecture:\\

\noindent {\bf Birational Section Conjecture.} (A. Grothendieck)
\emph{Let $K$ be a field finitely generated over $\mathbb{Q}$ or a finite extension of $\mathbb{Q}_p$ for some prime $p$. Then 
\begin{equation}
\Psi_{\overline{K}}: X(K) \rightarrow \mathcal{S}_{\overline{K}} \nonumber
\end{equation} 
is a bijection. In particular, the existence of a section of (3) implies the existence of a rational point on $X$.}\\

In \cite{koe3}, Koenigsmann establishes the local version of this conjecture, i.e. the case where $K$ is a finite extension of $\mathbb{Q}_p$. Later, Pop showed in \cite{popmeta} the even stronger result that $\Psi_{F''}:X(K) \rightarrow S_{F''}$ is a bijection, where $F''$ denotes the maximal elementary $\Z/p$ meta-abelian extension of $F$, with $F$ a finite extension of $\mathbb{Q}_p$ containing a primitive $p$-th root of unity.\footnote{He proves a slightly weaker result in the case when $F$ does not contain roots of unity which still implies Koenigsmann's original result.}

In this note we aim to show a result somewhere in between, namely that one can take $\hat{F}=F^{solv}$, the maximal solvable extension of $F$:\\

\begin{theorem}
Let $K$ be a finite extension of $\mathbb{Q}_p$. If $F^{solv}$ denotes the maximal solvable extension of $F$, then the map
\begin{equation}
\Psi_{F^{solv}}:X(K) \rightarrow \mathcal{S}_{F^{solv}} \nonumber
\end{equation}
is a bijection.\\
\end{theorem}

\noindent This follows from Pop's Theorem in the case where $K$ contains a primitive $p$-th root of unity. Pop's proof uses local-global principles for Brauer groups and uses crucially the fact that one is working with function fields of curves. The main novelty of this note is the method of proof, which goes instead via the following new group theoretic characterization of the existence of certain valuations on a field, of interest in its own right: \\

\begin{theorem}

Let $K$ be any field, $p$ a prime.
Then there is a valuation $v$ on $K$, extending uniquely to $K^{solv}$, with $\Gamma_v \not= p \Gamma_v$ and char$(Kv) \not= p$ if and only if $Gal(K^{solv}/K)$ has a non-procyclic $p$-Sylow subgroup with a non-trivial abelian normal subgroup. \\

\end{theorem}

\noindent To do this, we develop a general machinery of `canonical valuations' which allow one to deduce the section conjecture for any $\hat{F}$ satisfying certain technical properties. Roughly speaking, if the choice of $\hat{F}$ is such that one can develop a `good' notion of a $\hat{F}$-henselian valuation (i.e. a valuation on $F$ extending uniquely to $\hat{F}$), then we show that bijectivity of (2) is a purely formal consequence of the arguments from \cite{koe1} and \cite{ep}. Theorem 1.2 then follows from the fact that $F^{solv}$ satisfies the required properties.

In fact, pushing these arguments to their limit, we can even take $\hat{F}$ to be $F^{pq}$, the maximal $(p,q)$-meta-abelian extension of $F$, where $p$ and $q$ are two distinct primes and $F$ is a $p$-adic field containing a primitive $p$-th and $q$-th root of unity.\footnote{See Definition 8.6.}\footnote{Taking $\hat{F}=F^{pq}$ appears to be best possible using these methods, though see \cite{strommen_q2} for a strengthening when $F=\mathbb{Q}_2$.} The techniques here are based upon the fundamental characterization in \cite{koe2}. With $F$ as above, provided $F$ contains $\zeta_l$, where $l$ is a prime not equal to $p$, then one can show that
\begin{equation}
G_F(l) \simeq \Z_l \rtimes \Z_l. \nonumber
\end{equation}
where $G_F(l)$ denotes the maximal pro-$l$ quotient of $G_F$. In \cite{koe2} it is shown that any field $L$ with the same maximal pro-$l$ quotient must admit a $l$-henselian valuation so-called tamely branching at $l$. Since $l$ is arbitrary, it is clear that this criterion alone cannot recover a fully $p$-adic valuation. We will show that as soon as you add in some minimal knowledge of $p$-power extensions in the Galois group, you can recover it almost completely. In fact, we show that the maximal solvable quotient of the absolute Galois group recovers the valuation completely, while $Gal(F^{pq}/F)$, in the presence of roots of unity, recovers it up to some indeterminacy of the residue field. This gives a significant strengthening of the main result in \cite{koe1}.

Due to the strength of Theorem 1.2, we can also easily prove an analogue of the section conjecture for $p$-adic \emph{varieties} as well (Theorem \ref{bsc}) in this paper), stating that sections correspond to unique $K$-valuations.

\section{Preliminaries}

\subsection{Notation and Conventions}

Let $K$ be a valued field, with valuation $v$. Denote the valuation ring $\val_v$, the value group $\Gamma_v$ and the residue field $Kv$. If $a \in \val_v$, denote by $\bar{a}$ the image of $a$ in $Kv$. Given two valuation rings $\val_1$ and $\val_2$ on a field, $\val_2$ is said to be \emph{coarser} than $\val_1$ if $\val_1 \subset \val_2$. Two valuations are called \emph{comparable} if one is coarser than the other.
\\

\noindent Given a field $K$, let $G_K:=Gal(K^{sep}/K)$ denote the Galois group of a fixed separable closure of $K$. We have the following two important subfields of $K^{sep}$:
\begin{itemize}
\item $K(p)$, the maximal $p$-power extension of $K$, $p$ a prime. That is, the compositum of all extensions $L/K$ with $[L:K]=p^n$ for some $n$.
\item $K(p,q)$, the maximal $(p,q)$-extension of $K$, $p$ and $q$ distinct primes. That is, the compositum of all extensions $L/K$ with $[L:K]=p^nq^m$ for some $n,m$.
\item $K^{solv}$, the maximal solvable extension of $K$, i.e. the compositum of all extensions $L/K$ with $Gal(L/K)$ solvable.
\end{itemize}

\noindent The Galois groups $Gal(K(p)/K)$, $Gal(K(p,q)/K)$ and $Gal(K^{solv}/K)$ are naturally quotients of $G_K$. We denote them by $G_K(p)$ , $G_K(p,q)$ and $G_K^{solv}$ respectively.

\section{Some Galois Cohomology}

We recall some basics on Galois cohomology and the connection with Brauer groups and norms. The aim of this is to establish that the surjectivity of certain norm maps is a Galois theoretic property encoded by a very small quotient of $G_K$.

Let $p$ be a prime, $G$ a pro-$p$ group with rank $n$. Let $H^i(G):=H^i(G, \Z / p\Z)$, $i \in \mathbb{N}$ be the $i$-th Galois cohomology group.

If $K$ is a field and $L/K$ is a finite Galois extension, we let $N_{L/K}:L^{\times} \rightarrow K^{\times}$ denote the norm map. When $L=K(\sqrt[p]{a})$ for some $a \in K^{\times}$, we let $N(a)$ denote the image $N_{L/K}(L^{\times})$.

Now suppose $G=G_K(p)$, the maximal pro-$p$ quotient of $G_K$, is finitely generated, where $K$ is a field containing $\zeta_p$, a primitive $p$-th root of unity. Then Kummer Theory provides an isomorphism
\begin{equation}
H^1(G_K(p), \Z / p\Z) \simeq K^{\times}/(K^{\times})^p \nonumber
\end{equation}
and the theory of Brauer groups gives
\begin{equation}
H^2(G_K(p), \Z / p\Z) \simeq {_{p}Br(K)} \simeq (\Z/p\Z)^n \nonumber
\end{equation}
for some $n < \infty$, where ${_{p}Br(K)}$ is the $p$-torsion subgroup of the Brauer group of $K$. The cup-product pairing can be identified with the Hilbert symbol
\begin{equation}
K^{\times}/(K^{\times})^p \times K^{\times}/(K^{\times})^p \rightarrow (\Z / p\Z)^n \nonumber
\end{equation}
sending the pair $a,b$ to the symbol $(a,b)_K$ corresponding to the central simple $K$-algebra with generators $x,y$ subject to the relations $x^p=a, y^p=b, xy=\zeta_p yx$. We have $(a,b)_K=1$ iff $a \in N(b)$ iff $b \in N(a)$.

We will want to make use of a strengthening of the above observation. To this end we first make the following definition:

\begin{defn}\label{metadefn}
Given a field $K$ containing $\zeta_p$, let $K'$ denote the maximal $\Z/p\Z$ elementary abelian extension of $K$: thus $K'=K(\sqrt[p]{K^{\times}})$. Let $K''$ denote the maximal $\Z/p\Z$ elementary meta-abelian extension of $K$. That is, $K''=(K')'$.

If $G=G_K(p)$, we let $G':=Gal(K'/K)$, $G'':=Gal(K''/K)$.
\end{defn}

\begin{prop}\label{popbrauer}
Let $G=G_K(p)$ where $K$ is a field containing $\zeta_p$. Then
\begin{itemize}
\item[(i)] $H^1(G) \simeq H^1(G') \simeq K^{\times}/(K^{\times})^p$;
\item[(ii)] Given $a, b \in H^1(G)$, we have that $a \cup b = 0$ in $H^2(G)$ if and only if $a \cup b =0$ in $H^2(G'')$.
\end{itemize}
In particular, given $a,b$ in $K^{\times}/(K^{\times})^p$, whether or not $(a,b)_K$ is 1 or -1 can be read off $G''$. 
\end{prop}
\begin{proof}
Part (i) is just Kummer theory. For part (ii), see \cite{popmeta}, Lemma 1.
\end{proof}

We will also recall some basic facts about the cohomological dimension $cd(G)$ of a pro-$p$ group $G$.

\begin{prop}\label{cdlemma}
Let $G$ be a pro-$p$ group. Then
\begin{itemize}
\item[(i)] $cd(G) \leq n$ if and only if $H^{n+1}(G, \Z /p\Z)=0$;
\item[(ii)] $cd(G)=1$ if and only if $G$ is a free pro-$p$ group;
\item[(iii)] If $G=G_K(p)$, where $K$ is a field containing $\zeta_p$, then $cd(G)=1$ if and only if for every $a \in K^{\times} \setminus (K^{\times})^p$, the norm map
\begin{equation}
N_{L/F}:L^{\times} \rightarrow K^{\times} \nonumber
\end{equation} 
is surjective, where $L=K(\sqrt[p]{a})$. Equivalently, ${_{p}Br(K)}=0$.
\end{itemize}
\end{prop}
\begin{proof}
The first two items are standard (see \cite{serre2}). The last claim follows from the isomorphism $H^2(G_K(p), \Z / p\Z) \simeq {_{p}Br(K)}$ and the fact that ${_{p}Br(K)}$ is generated by the symbols $(a,b)_K$ (the Merkurjev-Suslin Theorem) which are trivial exactly when $b \in N(a)$.
\end{proof}

In fact, by Proposition \ref{popbrauer}, the conclusion of (iii) above holds even when $G$ is taken to be $G''$.

\subsection{Notions of Henselianity}

\begin{defn}
Let $H$ be a Galois extension of $K$, not necessarily finite. Then $(K,\val)$ is called $H$-henselian if $\val$ extends uniquely to $H$. Equivalently, if $\val$ extends uniquely to every finite sub-extension $K \subset L \subset H$.
\end{defn}

\begin{lemma}\label{henselslemma}(Hensels Lemma)
The following are equivalent:
\begin{itemize}
\item[(i)] $v$ is $H$-henselian;
\item[(ii)] Let $f \in \val_v[x]$ be a polynomial which splits in $H$. Then for every $a \in \val_v$ with $\bar{f}(\bar{a})=0$ and $\bar{f}'(\bar{a})\not=0$, there exists $\alpha \in \val$ with $f(\alpha)=0$ and $\bar{\alpha} = \bar{a}$.
\item[(iii)] Suppose the polynomial $x^n+x^{n-1}+a_{n-2}x^{n-2}+\ldots+a_0 \in \val_v[x]$, with $a_{n-2}, \ldots, a_0 \in \M_v$, splits in $H$. Then it has a zero in $K$.
\end{itemize}
\end{lemma}

\begin{rmk}
Note that given any valued field $(K,v)$, we can always find an $H$-henselization of it, that is, an extension $(K^h,v^h)$ of valued fields such that $v^h$ is $H$-henselian. 
\end{rmk}

\noindent The following choices of $H$ will be of crucial importance in the rest of this paper:
\begin{itemize}
\item $H=K^{sep}$. In this case we call an $H$-henselian valuation simply \emph{henselian}.
\item $H=K(p)$, the maximal $p$-power extension of $K$ for some prime $p$ (that is, the compositum of all Galois extensions of $K$ of degree $p^n$ for some $n$). In this case we call a $H$-henselian valuation \emph{$p$-henselian}.
\item $H=K$: the compositum of all Galois extensions of $K$ of degree $p^nq^m$. We call this the maximal $(p,q)$-extension of $K$. In this case an $H$-henselian valuation is called $(p,q)$-henselian.
\item $H=K^{solv}$: the maximal pro-solvable extension of $K$. In this case we call a $H$-henselian valuation \emph{solv-henselian}.
\end{itemize}

In the case of $p$-henselianity we have the following useful observation (see \cite{ep}, Theorem 4.2.2).

\begin{lemma}\label{phensel}
A valuation $v$ on a field $K$ is $p$-henselian if and only if it extends uniquely to every Galois extension of $K$ of degree $p$.
\end{lemma}

\section{Canonical classes}

\begin{defn}
Let $\cat$ be a class of finite groups closed under extensions, subgroups and quotients. If $G$ is a profinite group, we let $G^c$ denote the maximal pro-$\cat$ quotient of $G$. If $G=G_K$, we define $K^c$ to be the unique subextension of $K^{sep}$ with $G_K^c =Gal(K^c/K)$. For any field $K$, we let $\cat(K)$ denote the set of Galois subextensions of $K^c/K$.
\end{defn}

By Galois theory, the following properties are immediate:

\begin{itemize}
\item[(i)] If $L, F \in \cat(K)$ then the compositum $LF \in \cat(K)$;
\item[(ii)] If $L \in \cat(K)$ and $F/K$ is a subfield of $L$, then $F \in \cat(K)$;
\item[(iii)] If $L \in \cat(K)$ and $M \in \cat(L)$ then $M \in \cat(K)$;
\item[(iv)] $(K^c)^c=K^c$.
\end{itemize}

From now on $\cat$ will always refer to such a class.

\begin{defn}A valuation on $K$ which is $K^c$-henselian with respect to a class $\cat$ is called $\cat$-henselian or simply \emph{$c$-henselian}. We also say that $K$ is $c$-closed if $K=K^c$.
\end{defn}

We have the following proto-typical examples:

\begin{itemize}
\item $\cat=\cat_{sep}$, the class of all finite groups. Then $K^c=K^{sep}$ and $c$-henselianity is the same as henselianity.
\item $\cat=\cat_{p}$, the class of all $p$-groups. Then $K^c=K(p)$ and $c$-henselianity is the same as $p$-henselianity.
\item $\cat=\cat_{solv}$, the class of all solvable finite groups. Then we write $K^c=K^{solv}$, and call a $c$-henselian valuation \emph{solv-henselian}.
\end{itemize}

\begin{defn}
Let $\cat_1$ and $\cat_2$ be two classes. We say that $\cat_1$ \emph{contains} $\cat_2$ if, for any profinite group $G$, $G^{c_2}$ is obtained from $G^{c_1}$ as the quotient by a characteristic subgroup. Note that in this case, the class of finite groups in $\cat_1$ actually contains the finite groups in $\cat_2$.
\end{defn}
On the field-theory side, if $\cat_1$ contains $\cat_2$, then for any field $K$, if $L \in \cat_2(K)$, also $L \in \cat_1(K)$.\\

\begin{ex}
We have that $\cat_{solv}$ contains $\cat_p$ for any $p$. Indeed, let $G^p$ denote the maximal pro-$p$ quotient, and $G^s$ the maximal solvable quotient. Then $G^p$ is the quotient of $G^s$ by the normal subgroup generated by all its Sylow $q$ subgroups, $q \not=p$, which is characteristic (since any automorphism sends a Sylow $q$-subgroup to another Sylow $q$-subgroup).\\
\end{ex}

Since isomorphisms descend to quotients by characteristic subgroups, we get that if $\cat_1$ contains $\cat_2$, then
\begin{equation}
G_F^{c_1} \simeq G_K^{c_1} \Rightarrow G_F^{c_2} \simeq G_K^{c_2} \nonumber
\end{equation}
for any two fields $F$ and $K$.

\begin{defn}
Call a class $\cat$ \emph{canonical} if the following conditions hold for any field $K$.
\begin{itemize}
\item[(L)] Let $v$ be a $c$-henselian valuation on $K$, and assume $K^cv/Kv$ is separable. Then $K^cv \in \cat(Kv)$, and for any $L \in \cat(Kv)$, there exists a (not necessarily unique) $L' \in \cat(K)$ such that $L'w=L$, where $w$ is the unique extension of $v$. In particular, if $K^c=K$ then $(Kv)^c=(K^c)v=Kv$.
\item[(S)] If $\val_1$ and $\val_2$ are two independent $c$-henselian valuations on a field $K$, then $K=K^c$.
\item[(R)] If $K^c$ is a finite extension of $K$, then $[K^c:K] \leq 2$.\\
\end{itemize}
\end{defn}

\noindent {\bf Note:} From now on, when we refer to a class $\cat$ it will always refer to a canonical class, and a $c$-henselian valuation will always be with respect to some canonical class $\cat$. \\

\noindent As indicated (see \cite{ep} p.103 and \cite{koe2}), we have the following known result:\\

\begin{fact}
\emph{The classes $\cat_{sep}$ and $\cat_{p}$ are canonical.}\\
\end{fact}

\begin{rmk}
To explain condition (R), recall the following classical results: $K$ is real-closed (resp. Euclidean) if and only if $\overline{K}$ (resp. $K(2)$) is a finite, non-trivial extension of $K$, and in this case, the extension is of degree 2. Therefore this condition will allow us to keep close control over the behaviour of $K$ in the unusual cases where $K^c$ is a finite extension of $K$. \\
\end{rmk}

\noindent The following simple observation is crucial:

\begin{prop} \label{solvclass}
Let $\cat_{solv}$ be the class of solvable finite groups. Then $\cat_{solv}$ is a canonical class.
\end{prop}
\begin{proof}
Since $\cat_{solv}$ is closed under extensions, subgroups and quotients, it is a class in the sense of this paper. It remains to show that this class is canonical.

Suppose $v$ is a solv-henselian valuation on a field $K$. We need to show that we can lift solvable Galois extensions of $Kv$ to solvable Galois extensions of $K$. By Galois theory, the solvable Galois extensions are exactly the radical ones. Since $v$ is solv-henselian, and every Galois extension of degree $p$ is solvable, $v$ is $p$-henselian for every prime $p$, by Lemma \ref{phensel}. Since $\cat_p$ is canonical, any Galois extension of degree $p$ of $Kv$ can be lifted to $K$. Because any radical Galois extension can be written as a succession of extensions of prime degree, we can thus lift any radical Galois extension of $Kv$ to $K$. Note that for a field of characteristic $p$, a radical extension of degree $p$ is to be interpreted as an Artin-Schreier extension of degree $p$, i.e., an extension obtained by adjoining the roots of a polynomial of the form $x^p-x-a$. In the case when the valued field $(K,v)$ is of mixed characteristic $(0,p)$, assuming $\zeta_p \in K$, then such extensions of the residue field become `actual' radical extensions of $K$, namely the extension of degree $p$ obtained by adjoining a $p$-th root of $1+(\zeta_p-1)^pa$. Hence $\cat_{solv}$ satisfies property (L).

Next, suppose $v_1$ and $v_2$ are two independent solv-henselian valuations on a field $K$. As remarked in the above argument, $v_1$ and $v_2$ are in particular $p$-henselian for every prime $p$. Since $\cat_p$ is a canonical class, it follows that $K$ does not admit any non-trivial extensions of degree $p$. Hence it does not admit any radical extensions, whence $K=K^c$, implying that $\cat_{solv}$ satisfies property (S).

Finally, if $[K^c:K] < \infty$, then $[K(p):K]<\infty$ for every prime $p$. Again using that $\cat_p$ is canonical, we find that $K^c=K(2)$ and $[K(2):K] \leq 2$. Hence $\cat_{solv}$ satisfies (R).
\end{proof}

In an entirely analogous fashion we can prove the following:

\begin{prop}\label{pqclass}
Given two primes $p,q$, let $\cat_{p,q}$ be the class of finite $(p,q)$-groups, i.e., groups of cardinality $p^nq^m$ for some $n,m$. Then $\cat_{p,q}$ is a canonical class.
\end{prop}

We shall see that $\cat$-henselian valuations with respect to a canonical class $\cat$ admit a notion of a canonical $c$-henselian valuation. The first property we will need in this direction is that $c$-henselianity behaves well with respect to compositions of valuations. Indeed, let $v_1$ and $v_2$ be valuations on a field $K$ with valuation rings $\val_1$ and $\val_2$ respectively. If $\val_1 \subset \val_2$, so $v_2$ is a coarsening of $v_1$, we get an `exact sequence of valuations
\begin{equation}
1 \rightarrow v_2 \rightarrow v_1 \rightarrow v_1/v_2 \rightarrow 1
\end{equation}
Here $v_2/v_1$ is the induced valuation on $Kv_2$ with valuation ring $\overline{\val_1}:=\val_1 / \M_2$ and maximal ideal $\overline{\M_1}:=\M_1/\M_2$. The `lifting' property (L) is the key to the following

\begin{lemma}\label{exactsequence}
Given an exact sequence of valuations as above, then $v_1$ is $c$-henselian if and only if $v_2$ and $v_1/v_2$ is.
\end{lemma}
\begin{proof}
Suppose $v_1$ is $c$-henselian. Let
\begin{equation}
f = x^n + x^{n-1} + a_{n-2}x^{n-2} + \ldots + a_0 \in O_2[x] \nonumber
\end{equation}
be such that $a_i \in \M_1$. If $f$ splits in $K^c$ then since $\M_2 \subset \M_1$, Lemma \ref{henselslemma} implies that $f$ has a zero in $K$ and so $v_2$ is $c$-henselian. Next, assume that $\bar{a_i} \in \M_1 / \M_2$, and suppose $\bar{f}$ splits in $Kv_1^c=(K^c)v_1$. We may assume that $f$ splits in $K^c$. Indeed, without loss of generality suppose $f$ is irreducible. If $\bar{f}(\alpha)=0$ for $\alpha \in (Kv)^c$, then by property (L), there is an extension $F \in \cat(K)$ such that $Fv$ is the splitting field of $\bar{f}$. Then we can simply replace $f$ by the minimal polynomial of $a \in F$ with $\bar{a}=\alpha$. Hence $f$ has a zero in $\val_1$ by $c$-henselianity, and so also in $Kv_1$. Thus $v_2$ and $v_1/v_2$ are both $c$-henselian.

The other direction is straightforward. Let $f \in \val_1[x]$ be a polynomial which splits in $K^c$ and has a root in $Kv_1$. Then using $c$-henselianity of first $v_1/v_2$ and then $v_2$ one lifts the root first to  Hence $G_K^c=Gal(K^c/K)$. $\val_1/\M_1$ and then to $K$.
\end{proof}

\section{Constructing the canonical $c$-henselian valuation}

We mimic the classical construction.

\begin{defn}
Define subsets $C_1$ and $C_2$ of the set of all valuation rings of a field by
\begin{eqnarray}
&C_1&:=\{\val : \val \text{\, is $c$-henselian and $\val / \M$ is not $c$-closed} \} \nonumber \\
&C_2&:= \{ \val : \val \text{\, is $c$-henselian and $\val / \M$ is $c$-closed} \}. \nonumber
\end{eqnarray} 
If we want to emphasize the ambient field in question, we write $C_1(K)$, resp. $C_2(K)$.
\end{defn}

\noindent {\bf Note:} Since $K$ itself is always a $c$-henselian valuation ring of $K$, the set $C_1 \cup C_2$ is never empty.

\begin{rmk}\label{closedresfield}
Suppose that $\val \subset \val'$ are valuation rings of $K$ and $\val'$ has $c$-closed residue field $\val'/\M'$. Then by the Lifting Property (L), we know that the valuation ring $\val / \M'$ of $\val' / \M'$ has $c$-closed residue field $\val / \M$. Hence $\val$ also has a $c$-closed residue field.
\end{rmk}

Recall that two valuation rings $\val$ and $\val'$ are called `comparable' if one is a subset of the other.

\begin{prop}\label{canonicalconstruction}
Any two valuation rings from $C_1$ are comparable. If $C_2$ is non-empty, then $C_2$ contains a valuation ring which is coarser than every valuation ring from $C_2$ and strictly finer than every valuation ring from $C_1$. If $C_2$ is empty, then there is a finest valuation ring in $C_1$.
\end{prop}
\begin{proof}
We first show that two rings from $C_1$ are always comparable. Indeed, assume $\val_1, \val_2$ are incomparable $c$-henselian valuations. We will show that they are both in $C_2$, i.e. they have $c$-closed residue fields. It follows from the assumed incomparability that $\val := \val_1 \val_2$ is a proper coarsening of $\val_1$ and $\val_2$ and that the valuation rings $\val_1 / \M$ and $\val_2 / \M$ of $\val / \M$ are independent. Furthermore, by Lemma \ref{exactsequence}, they are both $c$-henselian. Thus by the (S)-property of $\cat$, $\val / \M$ is $c$-closed. By Remark \ref{closedresfield}, the residue fields of $\val_1$ and $\val_2$ are also $c$-closed: that is, they are in $C_2$.

Now, if $C_1$ is non-empty, then since all rings in $C_1$ are comparable, the intersection $\val^* := \bigcap_{\val \in C_1} \val$ is a valuation ring with maximal ideal $\bigcup_{C_1} \M$, which is clearly finer than every valuation ring in $C_1$. By Lemma \ref{henselslemma}, it is easy to see that $\val^*$ is $c$-henselian, so if $C_2 = \emptyset$, then $\val^*$ is a finest valuation ring in $C_1$, proving the last claim of the proposition.

Next suppose $C_2 \not = \emptyset$. Then a simple Zorn's Lemma construction shows that $C_2$ has a maximal element $\val^{**}$. Property (S) implies this element is unique. For supposing $\val_1$ and $\val_2$ are two distinct maximal elements, then their compositum $\val_3:=\val_1\val_2$ is $c$-henselian and $Kv_3$ has two independent $c$-henselian valuation rings $\val_1/\M_3$ and $\val_2/\M_3$. Hence $\val_3$ is in $C_2$, contradicting maximality.
\end{proof}

\begin{defn}
The \emph{canonical $c$-henselian valuation} of $K$, denoted by $\val_c$ (or $v_c$), is defined to be $\val^*$ if $C_2 = \emptyset$, and $\val^{**}$ otherwise.
We also put
\begin{equation}
C:= C_1 \cup \{\val_c \}. \nonumber
\end{equation} 
Thus the canonical $c$-henselian valuation is the finest valuation ring in $C$. If we want to emphasize the ambient field, we write $C(K)$.
\end{defn}

The point of this construction is that the canonical valuation enjoys many good structural properties not enjoyed by an arbitrary $c$-henselian valuation. The main such properties are summarized in the following

\begin{prop}\label{canonicalproperties}
The canonical $c$-henselian valuation satisfies the following properties.
\begin{itemize}
\item $\can$ is non-trivial if and only if $K$ is not $c$-closed and admits a non-trivial $c$-henselian valuation.
\item If $\val \in C$ then $\val$ is comparable to any other $c$-henselian valuation
\item If $\can$ does not have $c$-closed residue field, then neither does any other $c$-henselian valuation ring on $K$
\item If $\val$ is strictly coarser than $\can$, then $\val / \M$ is not $c$-closed. If $\val$ is finer than $\can$, it has $c$-closed residue field.
\item If $K$ is $c$-closed, then $C=\{K \}$.
\end{itemize}
\end{prop}
\begin{proof}
Follows easily from the construction. For example, for the second property, since $C_1$ and $C_2$ partition the set of $c$-henselian valuations, and $\val_c$ is comparable to every element in $C_1$ and $C_2$ by construction, it is comparable to every $c$-henselian valuation.
\end{proof}

\section{Three `Going-Down' results}

The formal properties of the canonical valuation are all that is required to prove the analogues of the three `Going-Down' theorems from \cite{ep} for $c$-henselian valuations. We prove the two we will need later and leave the third as an exercise to the reader. The proofs follow those in \cite{ep}.

\begin{prop}\label{normaldown}
Let $L \in \cat(K)$ be a normal extension, and suppose $\val' \in C(L)$. Then $\val := \val' \cap K \in C(K)$.
\end{prop}
\begin{proof}
If $L = L^c$ then $\val'=L$ and $\val = K$, and the claim is trivial. 

Suppose then that $L \not = L^c$ and $\val'$ is non-trivial. We will show that $\val'$ is the unique extension of $\val$ to $L$, and hence that $(K, \val)$ is $c$-henselian. Indeed, let $\val''$ be any extension of $\val$ to $L$. Then there is some $\sigma \in Gal(L/K)$ such that $\val'' = \sigma(\val')$. Hence $\val''$ is also $c$-henselian, and so by Proposition \ref{canonicalproperties}, is comparable, and hence equal to, $\val'$: indeed, distinct prolongations of a valuation to an algebraic extension are never comparable, by Lemma 3.2.8 in \cite{ep}.

We finally show that $\can(K) \subseteq \val$. Assume for a contradiction that $\val$ is strictly contained in $\can(K)$. By Prop \ref{canonicalconstruction}, $\val \in C_2(K)$. Now, by standard valuation theoretic arguments, we can find an extension $\val'''$ of $\can(K)$ to $L$ containing $\val'$: in particular, $\val'''$ contains $\can(L)$. In fact, it strictly contains it, since otherwise, upon restricting to $K$, we would get $\val = \can(K)$, contrary to assumption. Hence, by Prop. 1.11, $\val'''$ does not have $c$-closed residue field. Hence neither does $\can(K)$, implying $C_2(K)=\emptyset$, contradiction. 
\end{proof}

\begin{prop}\label{finitedown}
Suppose $L$ is not $c$-closed, and let $L \in \cat(K)$ be a finite extension. If $\val' \in C(L)$, then $\val := \val' \cap K \in C(K)$.
\end{prop}
\begin{proof}
One first passes to the normal hull of $L/K$, and then proceeds as above. 
\end{proof}

For the last Going-Down result, concerning Sylow $p$-extensions, we will need to add some extra technical conditions in the case $p=2$ (see \cite{ep} page 108-109). Recall (see \cite{ep} p. 109) that if there is a $c$-henselian valuation with real-closed residue field, then there exists a valuation ring $\val^{+} \in C_1(L)$ maximal with respect to the property of having a real-closed residue field.

\begin{prop} \label{sylowdown}
Let $L \in \cat(K)$ be a Sylow $p$-extension and let $\val' \in C(L)$. If $p=2$ and the residue field of $\val'$ is real-closed, we also assume $\val'$ is coarser than $\val^{+}$. Then $\val := \val' \cap K \in C(K)$.
\end{prop}
\begin{proof}
Assume $\val'$ is non-trivial, so $L \not= L^c$ by Proposition \ref{canonicalproperties}. Let $\val^c$ be the unique extension of $\val'$ to $L^c$. Now let $M \in \cat(L)$ be finite over $L$, and set $\val_1 = \val^c \cap M$, evidently a $c$-henselian valuation. We claim that $\val_1$ is the only $c$-henselian valuation ring of $M$ restricting to $\val$. 

Indeed, assume $\val_2$ is another such ring. Then we first claim $\val_1$ and $\val_2$ are not independent. Otherwise, $M = M^c$ by the (S) property, so $L^c = M$ is finite. By the (R) property, $[L^c:L]=2$, and since $L$ is the fixed field of a Sylow $p$-subgroup, $[M:L]=[L^c:L]=p^n$ for some $n$. It follows that $p=2$ and $L^c=L(2)$, so $L$, and hence also its residue field with respect to $\val'$, is real-closed. But note that we assumed $\val'$ was coarser than $\val^{+}$, and so $L$ cannot be real-closed: contradiction. Therefore $\val_1$ and $\val_2$ are not independent.

They are also incomparable, since they are distinct valuation rings both restricting to $\val$.
Hence $\val_3 := \val_1 \val_2$ is non-trivial, and its residue field $k := \val_3 / \M_3$ has independent valuations $\val_1 / \M_3$ and $\val_2 / \M_3$. Note that these valuations are $c$-henselian by Lemma \ref{exactsequence}. Hence, by the (S)-property, $k = k^c$. Now since $\val_1$ and $\val_2$ are non-comparable, $\val_1$ is a proper subset of $\val_3$, implying that $\val_3 \cap L$ is strictly coarser than $\val_c(L)$. Indeed, otherwise, upon restricting both to $L$, we find $\val' = \val_3 \cap L$, and since the former is $c$-henselian, this forces $\val_1 = \val_3$, contradicting the fact that $\val_1$ and $\val_2$ are not comparable. Hence $\val_3 \cap L$ does not have $c$-closed residue field $k''$. Since $[M:L]$ is finite, so is $[k:k'']$, with $k=(k'')^c$. It follows from the (R)-property that the degree of the extension is 2, and so by Lemma \ref{defect}, $2$ divides $[M:L]$. As $L$ is the fixed field of a Sylow $p$-subgroup, $[M:L]$ must be of degree $p^n$ for some $n$. This implies that $p=2$: in this case we have assumed that $\val'$ is coarser than $\val^{+}$. But then $\val''$ is strictly coarser than $\val^{+}$ and still has real-closed residue field, which gives a contradiction.

It is now straightforward to show that $\val$ is $c$-henselian, since it has a unique extension to $M$, which is itself $c$-henselian.
\end{proof}

\section{Rigid elements}

We recall the fundamental results from the theory of so-called `rigid elements'. This will be the key input to recover any sort of valuation whatsoever from the absolute Galois group. The theory developed above will then be used to bootstrap this valuation up to what we want.

Let $\val_v$ be a valuation ring of a field $K$. Then if $x \in K^{\times} \setminus \val_v^{\times}$, the ultrametric inequality implies the additive and multiplicative action of $\val_v^{\times}$ on $x$ possesses a certain rigidity, in the sense that one can never move too far away from $x$. Precisely, one has
\begin{equation}
\val_v^{\times} + x\val_v^{\times} \subseteq \val_v^{\times} \cup x\val_v^{\times}
\end{equation}

\noindent It turns out that any subgroup $T \leq K^{\times}$ which acts in a similarly rigid fashion on elements of $K^{\times} \setminus T$ must be induced by a valuation ring. 

\begin{defn}
If $x \in K^{\times} \setminus T$, then we call $x$ $T$-rigid if
\begin{equation}
T+xT \subseteq T \cup xT \nonumber
\end{equation}
\end{defn}

\noindent For simplicity we restrict now to the special case where $(F^{\times})^p \leq T$ for some prime $p$. In this case, define the sets
\begin{eqnarray}
\val_1(T)& :=& \{ x \in K \setminus T : 1+x \in T \} \nonumber \\
\val_2(T)& : =& \{x \in T : x \val_1(T) \subseteq \val_1(T) \nonumber \}
\end{eqnarray} and
\begin{equation}
\val(T) := \val_1(T) \cup \val_2(T). \nonumber
\end{equation}

\begin{prop}\label{rigidclassic}
Given the setup as above, suppose in addition that every element in $K^{\times} \setminus T$ is $T$-rigid, and if $p=2$, assume that $-1 \in T$. Then if $p \not=2$, $\val(T)$ defines a valuation ring of $K$ with $\val(T)^{\times} \subseteq T$. If $p=2$, there exists a subgroup $T' \leq K^{\times}$ containing $T$ such that $[T':T]=2$ and $\val(T')$ is a valuation ring of $K$ with $\val(T')^{\times} \subseteq T'$.
\end{prop}
\begin{proof}
This is Theorem 2.2.7 in \cite{ep}.
\end{proof}

\noindent So provided $p \not=2$, the valuation ring will be non-trivial if and only if $T \not= K^{\times}$.

The next lemma gives a powerful method for detecting the existence of subgroups $T$ satisfying the criterion of proposition \ref{rigidclassic}.

\begin{lemma}\label{wonderful}
Let $p$ be an odd prime, $K$ a field. Suppose $S \leq K^{\times}$ is a subgroup of index $[K^{\times}:S] \geq p^2$, such that for any $x \in K^{\times} \setminus S$,
\begin{equation}
S+xS \subseteq \bigcup_{i=0}^{p-1} x^iS. \nonumber
\end{equation}
Then there is a subgroup $T \leq K^{\times}$ with $S \subseteq T$, $[T:S] \leq p$, and every $x \in K^{\times} \setminus T$ is $T$-rigid.
\end{lemma}
\begin{proof}
This is Lemma 2.14 in \cite{koe2}.
\end{proof}

For later use, we also make the following definition:

\begin{defn}
Given a field $K$ and a prime $p$, an element $a \in K \setminus K^p$ is called \emph{strongly $p$-rigid} iff it is $(K^{\times})^p$-rigid, i.e., iff
\begin{equation}
K^p + aK^p \subseteq K^p \cup aK^p. \nonumber
\end{equation}
\end{defn}

Proposition \ref{rigidclassic} shows sufficiently many strongly $p$-rigid elements induce the existence of a non-trivial valuation ring. In fact, in \cite{koe3} it was shown, using model theory, that even just a single strongly $p$-rigid element already implies the existence of such a valuation.

\section{A Galois-theoretic characterization of $c$-henselianity}

A Galois theoretic characterization for a field to admit a non-trivial $p$-henselian valuation was obtained in \cite{koe1}, provided the field contains a primitive $p$-th root of unity $\zeta_p$. The formal properties of canonical valuations established above allow us to obtain an analogous characterization for the existence of a $c$-henselian valuation in terms of the maximal $\cat$-quotient of the absolute Galois group.

\begin{defn}
A valuation $v$ on a field $K$ is said to be \emph{tamely branching} at the prime $p$ if $char(Kv) \not =p$, $\Gamma_v \not = p\Gamma_v$. If $[\Gamma_v: p\Gamma_v]=p$, we also require that $Kv$ is not $p^2$-closed, that is, there exists a separable extension of $Kv$ of degree divisible by $p^2$.
\end{defn}

\noindent Notice that if $p=2$ and $Kv$ is formally real, then $Kv$ admits an extension of degree 2 but not degree 4, as $[Kv^{sep}:Kv]=2$. This is however the only case for which having an extension of degree $p$ does not imply that there is also an extension of degree $p^2$. So outside of this case, the last condition is equivalent to $Kv$ not being $p$-closed.

The following observation will be crucially used later.

\begin{lemma}\label{tamely}
Let $F$ be a finite extension of $\mathbb{Q}_p$. Then $F$ does not admit any $p$-henselian valuation tamely branching at $p$.
\end{lemma}
\begin{proof}
Let $v_p$ denote the $p$-adic valuation, and suppose $w$ is another valuation which is $p$-henselian tamely branching at $p$. As $v_p$ is a rank 1 valuation and has a residue field which is not $p$-closed, $w$ must be a refinement of $v_p$, and hence must have residue characteristic $p$: contradiction.
\end{proof}

We now present a sharpening of the Galois-characterization for $p$-henselian valuations tamely branching at $p$ obtained in \cite{koe1}. Recall Definition \ref{metadefn} of the maximal elementary $\Z/p\Z$ meta-abelian extension. Let us also recall that if $F = \mathbb{Q_l}(\zeta_p)$, $l not= p$, then one can show that $G_F(p) \simeq \Z_p \rtimes \Z_p$, and so the maximal elementary $\Z/p$ meta-abelian quotient is $\simeq \Z / p^2\Z \rtimes \Z / p^2\Z$. For this field $F$ it is also known that the norm maps $N_{L/F}$ are not surjective when $L=F(\sqrt[p]{a})$. By Lemma \ref{cdlemma}, the same will therefore be true of any other field $K$ for which the maximal elementary $\Z/p$ meta-abelian quotientof $G_K$ is of the same form.

\begin{prop}\label{sharpening}
Let $p$ be a prime, and let $K$ be a field with a primitive $p$-th root of unity. Let $K^{''}$ denote the maximal elementary $\Z/ p\Z$ meta-abelian extension of $K$. Then $K$ admits a $p$-henselian valuation tamely branching at $p$ whenever $Gal(K^{''}/K) \simeq \Z / p^2\Z \rtimes \Z / p^2\Z$.
\end{prop}
\begin{proof}
We will only treat the case $p>2$ in what follows. The case $p=2$ follows using the same method as in \cite{ep} Lemma 5.4.4. 

Let us suppose first of all that $G:=Gal(K^{''}/K) \simeq \Z / p^2\Z \rtimes \Z / p^2 \Z$. Then by Kummer theory,
\begin{equation}
\text{dim}_{\mathbb{F}_p} K^{\times} / (K^{\times})^p = \text{rank}(G)=2. \nonumber
\end{equation}
Suppose $H \leq G$ is a subgroup of index $p$. Then we claim that $H \simeq \Z / p^i \Z \rtimes \Z / p^j\Z$ where $i,j \in \{1,2\}$. Indeed, the embedding $H \hookrightarrow \Z / p^2\Z \rtimes \Z / p^2 \Z$ induces the following commutative diagram with exact rows:
\begin{center}
\begin{tikzpicture}[>=angle 90]
\matrix(d)[matrix of math nodes, row sep=3em, column sep=2em, text height=1.5ex, text depth=0.25ex]
{1 & \Z /p^2\Z & G & \Z /p^2\Z & 1 \\
1 & H'' & H & H' & 1\\};
\path[->, font=\scriptsize]
 (d-1-1) edge node[above]{} (d-1-2)
 (d-1-2) edge[right hook->] node[above]{$g$} (d-1-3)
 (d-1-3) edge[->>] node[above]{$f$} (d-1-4)
 (d-1-4) edge node{} (d-1-5)
 (d-2-1) edge node[above]{} (d-2-2)
 (d-2-2) edge[right hook->] node{} (d-2-3)
 (d-2-3) edge[->>] node[above]{} (d-2-4)
 (d-2-4) edge node{} (d-2-5)
 (d-2-2) edge[right hook->] node{} (d-1-2)
 (d-2-3) edge[right hook->] node{} (d-1-3)
 (d-2-4) edge[right hook->] node[right]{} (d-1-4);
\end{tikzpicture}
\end{center}
where $H''=im(g) \cap H$, $H'=f(H)$, and since $H'$ is cyclic, the splitting of the top sequence induces one for the bottom one. So $H \simeq H' \rtimes H''$. If $H'$ or $H''$ were trivial, then $H$ would have index greater than $p$, contradicting our assumption.

If $L$ is an extension of $K$ of degree $p$, applying the above in the case when $H = Gal(K^{''}/L)$, we see that
\begin{equation}
\text{dim}_{\mathbb{F}_p} L^{\times} / (L^{\times})^p = \text{rank}(H)=2 
\end{equation}
as well. Armed with this crucial observation, we now wish to use Lemma \ref{wonderful} with $S=(K^{\times})^p$. 

If we let
\begin{equation}
\langle x \rangle := \bigcup_{i=0}^{p-1} x^i (F^{\times})^p \nonumber
\end{equation}
then we need to show that for every $x \in F^{\times} \setminus (F^{\times})^p$, $(F^{\times})^p+x(F^{\times})^p \subset \langle x \rangle (F^{\times})^p$, where this last set denotes the multiplicative group generated by $x$ and $(F^{\times})^p$. Notice that for any $a, b \in F^{\times}$, $z:=a+\sqrt[p]{x}b \in L:=K(\sqrt[p]{x})$ has norm $N_{L/F}(z)=a^p+xb^p$. Therefore the conditions of Lemma \ref{wonderful} are met if we can show that
\begin{equation}
N_{L/F}(L^{\times}) = \langle x \rangle (F^{\times})^p
\end{equation}
for any such $L$. Since $N_{L/F}(\sqrt[p]{x})=x$, we have that $\langle x \rangle (F^{\times})^p \subset N_{L/F}(L^{\times})$ and, since $x \not \in F^p$, $\sqrt[p]{x} \not \in L^p$. By Lemma \ref{popbrauer} and the discussion preceeding the statement we are proving, $N_{L/F}:L \rightarrow F$ is not surjective. Thus we may find $y \in F^{\times} \setminus \langle x \rangle (F^{\times})^p$. Since $L^p \cap F = F^p$, $y \not \in \langle \sqrt[p]{x} \rangle (L^{\times})^p$. Thus $y$ and $\sqrt[p]{x}$ are independent elements in the $\mathbb{F}_p$-vector space $L^{\times}/(L^{\times})^p$, which is 2-dimensional by (2.3). Thus
\begin{equation}
L^{\times} = \langle y \rangle \langle \sqrt[p]{x} \rangle (L^{\times})^p \nonumber
\end{equation}
from which, by taking norms, we obtain (2.4) as desired.

Since $[\Gamma:p\Gamma]=[K^{\times}:(K^{\times})^p]=p^2$, we can use Lemma \ref{wonderful} together with Proposition \ref{rigidclassic} to see that $K$ admits a valuation $\val$ with $\val^{\times} \leq T$, for some $T \subsetneq F^{\times}$ containing $(F^{\times})^p$. Since then $\val^{\times}(F^{\times})^p \subset T \not = F^{\times}$, we have $\Gamma_v \not = p\Gamma_v$. 

The rest of the proof now proceeds exactly as in \cite{koe1}.
\end{proof}

We are now ready to deduce the first main result, giving a Galois theoretic characterization for a field to admit a $c$-henselian valuation. We will assume in the proof that $p>2$ for simplicity.

\begin{theorem}\label{theorem1}
Let $K$ be any field, and let $\cat$ be a canonical class containing $\cat_p$ for some prime $p$. If $K(\zeta_p) \not \in \cat(K)$, then we will also assume that $\zeta_p \in K$.
Then there is a $c$-henselian valuation $v$ on $K$ tamely branching at $p$ if and only if $Gal(K^c/K)$ has a non-procyclic $p$-Sylow subgroup with a non-trivial abelian normal subgroup. 
\end{theorem}
\begin{proof}
(``$\Rightarrow$''): If $K$ admits such a valuation, then the valuation extends to a $c$-henselian valuation on the fixed field of \emph{any} $p$-Sylow subgroup $S$ of $Gal(K^c/K)$. By Hilbert ramification theory, the inertia subgroup of this extended valuation is a non-trivial normal abelian subgroup, and $S$ is not procyclic.

(``$\Leftarrow$''): Let $S$ be such a Sylow subgroup, with fixed field $F$. By assumption $S$ admits a non-trivial abelian normal subgroup $A \simeq \mathbb{Z}_p^r$ where $r=rank(A)$. 

If $r>1$, then $S$ has a normal subgroup of the form $\Z_p \rtimes \Z_p$, and so its fixed field is a \emph{normal} field extension $L/F$ inside $K^c$ such that
\begin{equation}
Gal(K^c/L) = Gal(L^c/L) \simeq \mathbb{Z}_p \rtimes \mathbb{Z}_p. \nonumber
\end{equation}
Since
\begin{equation}
cd_p(\mathbb{Z}_p \rtimes \mathbb{Z}_p)=2 \nonumber
\end{equation}
we find $char(L) \not= p$, since the $p$-cohomological dimension of a field of characteristic $p$ is always $\leq 1$ (see \cite{serre2} Chapter 2, Section 2.2). By construction\footnote{Noting that since $\cat$ contains $\cat_p$, $K^c$ is $p$-closed.}, $L(p)=L^c=K^c$, and $Gal(L^c/L)(p)=Gal(L(p)/L) \simeq \mathbb{Z}_p \rtimes \mathbb{Z}_p$. If $K(\zeta_p) \not \in \cat(K)$, then $\zeta_p \in K \subset L$ by assumption. On the other hand, suppose $K(\zeta_p) \in \cat(K)$. We have that $[K(\zeta_p):K]$ divides $p-1$. Since $\zeta_p \in K^c$, but there are only $p$-power extensions between $L$ and $K^c$, it must therefore be that $\zeta_p \in L$. Thus in all cases, $\zeta_p \in L$. Then by Proposition \ref{sharpening}, there is a $p$-henselian valuation $w$ on $L$ tamely branching at $p$. Since $L(p)=L^c$, the valuation is actually $c$-henselian. Let $v$ be the canonical $c$-henselian valuation on $L$. By Proposition \ref{canonicalproperties}, $v$ is still tamely branching at $p$. By Proposition \ref{normaldown}, its restriction to $F$ is again $c$-henselian, and clearly still has residue characteristic not $p$ and value group not $p$-divisible. Finally, by Proposition \ref{sylowdown}, we may once more restrict to $K$ and obtain a $c$-henselian valuation tamely branching at $p$ as desired.

If $r=1$, then since $S$ is not pro-cyclic, there is $g \in S \setminus A$ such that
\begin{equation}
A \rtimes \langle g \rangle \simeq \Z_p \rtimes \Z_p. \nonumber
\end{equation}
Letting $L$ be the fixed field of this semidirect product, we find in the same way as above that $L$ has a $c$-henselian valuation tamely branching at $p$. Its unique prolongation $w$ to the fixed field $Fix(A)$ of $A$ has non-$p$-divisible value group and residue characteristic not $p$, and so the same will be true for $w_c$, the canonical $c$-henselian valuation on $Fix(A)$. By the `Going-Down' results, the restriction of $w_c$ to $L$ is $c$-henselian and tamely branching, and therefore so is its restriction to $F$, which gives us the desired valuation.
\end{proof}

\begin{rmk}
For example, we may take $\cat$ to be $\cat_{solv}$ in the above. Since $K(\zeta_p)$ is a solvable extension, we do not in this case need to assume anything about $K$ containing $\zeta_p$.
\end{rmk}

We record the following strengthening of the above utilizing the full sharpening obtained in Proposition \ref{sharpening}. 

\begin{defn}\label{pq}
We denote by $K^{pq}$ the compositum of all elementary abelian $\Z/q\Z$ extensions of $K''$, the elementary $\Z/p\Z$ meta-abelian extension of $K$. We call $K^{pq}$ the maximal $(p,q)$-meta-abelian extension of $K$.
\end{defn}

\begin{cor}\label{pqmeta}
Let $K$ be any field containing $\zeta_p$ and let $\cat=\cat_{p,q}$, the class of all finite groups of order $p^nq^m$ for some $n,m$. Then there is a $c$-henselian valuation $v$ on $K$ tamely branching at $p$ if and only if $Gal(K^{pq}/K)$ has a non-procyclic $p$-Sylow subgroup with a non-trivial abelian normal subgroup.
\end{cor}
\begin{proof}
This follows in the exact same way as the proof of the above Theorem.
\end{proof}



\section{Recovering the $p$-adic valuation}

Now let $\cat$ be any canonical class containing $\cat_{p,q}$: for example $\cat_{p,q}$ or $\cat_{solv}$ . Thus Theorem \ref{theorem1} can be applied in this context. We will now show that if we impose extra structure on the groups in question, we are rewarded with extra structure on the valuations obtained in this way. We will first need some preliminary technical results.

\begin{prop}\label{residuefield}
Let $(K,v)$ be a valued field of mixed characteristic $(0,p)$ such that $\mathcal{O}[1/p]=K$, $K^{\times}/(K^{\times})^p$ is finite. Then $Kv$ is perfect. If in addition $\Gamma_v \not = p\Gamma_v$, then $\Gamma_v \simeq \mathbb{Z}$ and $Kv$ is a finite field.
\end{prop}
\begin{proof}
This is just \cite{popmeta} Lemma 2.4.
\end{proof}

The proof of the next proposition was related to the author by Koenigsmann.

\begin{prop}\label{tricky}
Let $(K,v)$ be a $p$-henselian valued field of mixed characteristic $(0,p)$ with $\mathcal{O}_v[1/p]=K$ and suppose that $G_K(p)$ is finitely generated. Then
\begin{equation}
\Gamma_v = p\Gamma_v \implies cd(G_K(p)) \leq 1. \nonumber
\end{equation}
\end{prop}
\begin{proof}
By Lemma \ref{cdlemma}, it suffices to show that for any $a \in K^{\times} \setminus (K^{\times})^p$,
\begin{equation}
N_{L/K}:L^{\times} \rightarrow K^{\times} \nonumber
\end{equation}
is surjective, where $L=K(\sqrt[p]{a})$. Since $\Gamma_v =p\Gamma_v$, $K^{\times} = \val_v^{\times}(K^{\times})^p$, and so $a$ may be taken to be a unit. Since the residue characteristic is a perfect field of characteristic $p$ by Proposition \ref{residuefield}, we may further take $a$ to be in $1+\M$, say $a=1+y$.

Let $\alpha$ be a primitive element for $L/K$ which is integral and has trace 1. Now we claim that there is $x \in K$ such that
\begin{equation}
N_{L/K}(1+x\alpha)=1+y.\nonumber
\end{equation} 
Indeed, expanding the norm we get $N(\alpha)x^p+ \ldots - x+1=1+y$. Let $f(x) \in \val[x]$ be $N(\alpha)x^p+ \ldots -x-y$: we need to show that $f$ admits a root in $K$. But $\bar{f}(\bar{y})=0$ since $v(y)>0$. This root is furthermore simple, since $f'(y)=y(ay^{p-1}+\ldots)-a$ which becomes $-1$ in the residue field. By $p$-henselianity, this root lifts to $K$ as desired.
\end{proof}

In fact, it is possible to prove the following even stronger result, though we omit its proof as its full strength is not necessary for our considerations.

\begin{prop}
Let $(K,v)$ be a $p$-henselian valued field of mixed characteristic $(0,p)$ with $\mathcal{O}_v[1/p]=K$ and suppose that $G_K(p)$ is finitely generated. Then if $\Gamma_v =p\Gamma_v$, there exists a field $F$ of characteristic $p$ such that $G_K(p) \simeq G_F(p)$.
\end{prop}

The last result we need is a strengthening of a lemma by Pop (Satz 4 of \cite{pop_kenn}). We simply optimize his original proof. 

\begin{lemma}\label{pop}
Let $G:=Gal(F^{pq}/F)=G_F^{pq}$ where $F$ is a finite extension of $\mathbb{Q}_p$ and $F^{pq}$ is as in Definition \ref{pq}. Then there is a $p$-subgroup $R$ of $G$ such that if $H \mathrel{\unlhd} G$ is non-trivial, then $H \cap R \not = \{1\}$.
\end{lemma}
\begin{proof}
Let $I_F$ and $R_F$ denote the inertia and ramification subgroup of $G$ with respect to the $p$-adic valuation on $F$. We claim that $R_F$ satisfies the desired property. 

Indeed, suppose $H$ is any normal subgroup, and let $L$ be the fixed field of $H$ in $F^{pq}$. Then note that $R_F \cap H = R_L$, the ramification subgroup of the $p$-adic valuation on $L$. So we need to show that this ramification group is non-trivial. We will show that the $p$-Sylow subgroups of $G_L^{pq}$ are non-cyclic. Assuming this, note that if $R_L=1$, then $I_L \simeq (\Z/q)^r$ for some $r$. Also, $G_L^{pq}/I_L \simeq G_{Lv}^{pq}$. The Sylow subgroups of $G_L^{pq}/I$ are of the form $PI/I$ where $P$ is a Sylow subgroup of $G_L^c$. Since $I_L$ is not pro-$p$ (and has no pro-$p$ quotients) it commutes with any Sylow subgroup $P$ as both are normal. Thus $PI/I$ is cyclic if and only if $P$ is cyclic. But $G_{Lv}^{pq} \simeq \Z/p \times \Z/q$ clearly has a cyclic $p$-Sylow subgroup, which gives us our desired contradiction.

Let $F_1/F$ be any Galois sub-extension of $F^{pq}$ not contained in $L$, and put $k=F_1 \cap L$, $L_1=F_1^{'} \cap L$, where $F_1^{'}$ is the maximal elementary abelian $\Z/p$-extension of $F_1$. Since $L_1$ and $F_1$ are linearly disjoint, 
\begin{equation}
Gal(L_1/L)^c \simeq Gal(L_1F_1/F_1)^c \nonumber
\end{equation}
and $Gal(L_1F_1/F_1)$ is a quotient of $Gal(F_1^{'}/F_1)$. Therefore $L^{'}_1/k$ is a $\Z/p\Z$-extension. Now
\begin{equation}
[F_1^{'}:L_1F_1]=\frac{[F_1^{'}:F_1]}{[L_1:k]}\geqslant p^{a-b} \nonumber
\end{equation}
where $a=[F_1:\mathbb{Q}_p], b=[k:\mathbb{Q}_p]$. By taking an element $\alpha$ in $F^{pq}$ of degree $p^2q$ over $F$ but not contained in $L$, we may choose $F_1=F(\alpha)$. Since $L/F$ is of degree at most $p^2q$, $[F_1:k]$ is at least degree $p$ or $q$, and in either case is at least 2. Then $a-b \geqslant 2$ by the Tower Law, and so $p^2 \mid [F_1^{'}:L_1F_1]$. It follows that $Gal(F_1^{'}/L_1F_1)^c$ is at least $(\Z/p\Z)^2$ and so is not cyclic.

Now, any $p$-Sylow subgroup of $Gal(F_1^{'}/L_1)^{pq}$ must contain $Gal(F_1^{'}/L_1F_1)^{pq}$, as it's a subgroup of the $p$-group $Gal(F_1^{'}/F_1)^c$. Because any subgroup of a cyclic group is cyclic, it follows that $Gal(F_1^{'}/L_1)^{pq}$ has no cyclic $p$-Sylow subgroups. Since $Gal(F_1^{'}/L_1)^c \simeq Gal(LF_1^{'}/L)^{pq}$, neither does $Gal(LF_1^{'}/L)^{pq}$. But as this is a quotient of $G_L^{pq}$, it follows that $G_L^{pq}$ also cannot have any cyclic $p$-Sylow subgroups.
\end{proof}

Armed with the above technicalities, we are ready to prove the second main result.

\begin{theorem}\label{keyprop}
Let $F$ a finite extension of $\mathbb{Q}_p$ containing $\zeta_p$ and $ \zeta_q$ with $p$-adic valuation $v_p$. Choose $\cat$ to be any canonical class containing $\cat_{p,q}$, where $q$ is any prime different from $p$. Suppose $L$ is any field with
\begin{equation}
G_L^{c} \simeq G_F^c, \nonumber
\end{equation}
where, if $L(\zeta_n) \not \in \cat(L)$, $n \in \{p,q\}$, we additionally assume that $\zeta_n \in L$.
Then there is a $c$-henselian valuation $v$ on $L$ with $Lv$ a finite field of characteristic $p$ and $\Gamma_v \simeq \mathbb{Z}$. Furthermore, there is a finite extension $F'$ of $\mathbb{Q}_p$ with $p$-adic valuation $v_p$, such that $G_{F'}^c \simeq G_{F}^c$, $[F':\mathbb{Q}_p]=[F:\mathbb{Q}_p]$, and $Lv \simeq F'v_p$. If we take $\cat=\cat_{solv}$ then $F'$ can be taken to be $F$.
\end{theorem}
\begin{proof}
Let $v$ be the finest non-trivial $c$-henselian valuation on $L$, which exists by Theorem \ref{theorem1}. Let us first show that the residue characteristic of $v$ is $p$.

Suppose, for a contradiction, that the residue characteristic is not $p$. If $\Gamma_v \not= p\Gamma_v$ then $L$ contains strongly $p$-rigid elements: indeed, it is not hard to see that any $a$ with $v(a) \not \in p\Gamma_v$ is strongly $p$-rigid. By the main result of \cite{koe2}, $L$ therefore admits a $p$-henselian valuation tamely branching at $p$, which by Proposition \ref{sharpening} is encoded in $G_L(p)$. The isomorphism $G_L^c \simeq G_F^c$ forces their maximal pro-$p$ quotients to be isomorphic, and since we are assuming $\cat$ contains $\cat_{p,q}$, these coincide naturally with the maximal pro-$p$ quotients of the full absolute Galois group. It follows, again by Proposition \ref{sharpening}, that $F$ also admits a $p$-henselian valuation tamely branching at $p$, contradicting Lemma \ref{tamely}.

Hence it must be that $\Gamma_v = p \Gamma_v$. Because $char(Lv)\not=p$, the inertia subgroup $I_v$ of $G_L^{c}$ is normal and contains no non-trivial pro-$p$ subgroups. By Lemma \ref{pop}, this forces $I_v$ to be trivial. Hence
\begin{equation}
G_{Lv}^{c} \simeq G_L^{c}/I_v \simeq G_F^c. \nonumber
\end{equation}
Again by Theorem \ref{theorem1}, $Lv$ admits a non-trivial $c$-henselian valuation, from which we may obtain a proper refinement of the original valuation on $L$, contradicting the fact that we choose $v$ to be the finest such. Thus it must have been the case that char$(Lv)=p$.

Now, since $G_L(p) \simeq G_F(p)$ as remarked above, and $cd(G_F(p))=2$, Proposition \ref{tricky} implies that $\Gamma_v \not = p\Gamma_v$. Since a $p$-adic field has small absolute Galois group, having only finitely many extensions of a given degree, we may apply Proposition \ref{residuefield} to deduce that $\Gamma_v \simeq \mathbb{Z}$, and that $Lv$ is a finite field of characteristic $p$.

Put $L':=L \cap \overline{\mathbb{Q}}$ and let $F'$ be the henselization of $L'$ with respect to $v'$, the restriction of $v$ to $L'$. The induced valuation on $L^h$ still has value group $\mathbb{Z}$ and residue field finite of characteristic $p$: therefore it is a finite extension of $F$ and $v'$ coincides with the $p$-adic valuation $v_p$. By construction,
\begin{equation}
G_{F'}^c \simeq G_L^c \simeq G_F^c \nonumber
\end{equation}
and $F'v_p \simeq Lv$. Since $G_{F'}^c \simeq G_F^c$, we have $G_{F'}(p) \simeq G_F(p)$. By \cite{serre2} Section 5.6 Lemma 3, we must have that $[F':\mathbb{Q}_p]=[F:\mathbb{Q}_p]$.

Suppose next that $\cat=\cat_{solv}$. Then by work of Jarden, Ritter and Jenkner (\cite{jardenritter}, \cite{ritter}, \cite{jenkner}), it follows that 
\begin{equation}
F' \cap \mathbb{Q}_p^{ab} = F \cap \mathbb{Q}_p^{ab}, \nonumber
\end{equation}
which forces $Lv = Fv_p$.
\end{proof}

\noindent Note that as before, if we take $\cat=\cat_{solv}$, then we do not need any extra assumptions on $L$ containing roots of unity.

Let us also observe that from the above proof it follows that a minimal positive element in $\Gamma_v$ above may be taken to be $v(\pi)$ where $\pi$ is a uniformizer of $F'$ algebraic over $\mathbb{Q}$. Indeed, the subgroup of $\Gamma_v$ generated by $v(\pi)$ will already be all of $\Z$.

\begin{cor}\label{pqmeta}
Let $F$ be a finite extension of $\mathbb{Q}_p$ containing $\zeta_p$ and $\zeta_q$. If $L$ is any field also containing $\zeta_p$ and $\zeta_q$, and if $Gal(L^{pq}/L) \simeq Gal(F^{pq}/F)$, then $L$ admits a non-trivial $(p,q)$-henselian valuation $v$ with $\Gamma_v \simeq \Z$. Furthermore, there is a finite extension $F'$ of $\mathbb{Q}_p$ containing $\zeta_p$ and $\zeta_q$ such that $G_{F'}^{pq} \simeq G_F^{pq}$, $[F':\mathbb{Q}_p]=[F:\mathbb{Q}_p]$ and $Lv \simeq F'v_p$.
\end{cor}
\begin{proof}
The proof is identical to the above, simply using Corollary \ref{pq}. 
\end{proof}

\section{The Section Conjecture}

We are now ready to prove the main result.

\begin{theorem}\label{bsc}
Let $X$ be a smooth, projective variety of dimension $n$, where $F$ is a finite extension of $\mathbb{Q}_p$ and $F$ contains $\zeta_p$ and $\zeta_q$. Then given any section $s$ of
\begin{equation}
1 \rightarrow G_{\overline{F}(X)}(p,q) \rightarrow G_{F(X)}(p,q) \rightarrow G_{F}(p,q) \rightarrow 1
\end{equation}
there exists a unique $F$-valuation $v$ of $F(X)$ such that $s$ lies above $v$. In particular, the existence of a section implies the existence of a point. When $X$ is a curve, the $F$-valuation is induced by a unique point $a \in X(F)$ and therefore the section lies over $a$.
\end{theorem}
\begin{proof}
Let $s:G_F(p,q) \rightarrow G_{F(X)}(p,q)$ be a section, and let $K$ be the fixed field in $F(X)(p)$ of $s(G_F(p,q))$. Then $G_K(p,q) \simeq s(G_F(p,q)) \simeq G_F(p,q)$. By Theorem \ref{keyprop} there is a finite extension $F'/ \mathbb{Q}_p$ and a valuation $v$ on $K$ with value group $\Z$ and residue field isomorphic to $F'v_p$, where $v_p$ is the $p$-adic valuation on $F'$. Let $\pi$ be a uniformizer of $F'$ with respect to $v_p$ which is algebraic over $\mathbb{Q}$. Then $v(\pi)$ is a minimal positive element in $\Gamma_v$. Consider the restriction $w$ of $v$ to $F'(X)$. Then $w$ still has residue field $F'v_p$ and $w(\pi)$ is still minimal positive. Let $H$ be the subgroup of $\Gamma_w$ generated by $w(\pi)$.

Since $F'$ is complete, it admits no immediate extensions of transcendence degree $n$. Therefore $H \not= \Gamma_w$. Let $w'$ be the valuation obtained from $w$ with value group $\Gamma_w / H$. By construction, $w'$ is trivial on $F'$ and has residue field $F'$, since $w'(\pi)=0$. Since $w'$ is a coarsening of a $p$-henselian valuation, it is itself $p$-henselian. Hence $w'$ is an $F'$-valuation with $s(G_F(p)) \subset D_{w'}$. 

To show uniqueness, suppose $w''$ is another valuation such that $s(G_{F'}(p,q)) \subset D_{w''}$. Then as both are $p$-henselian with residue field not $p$-closed, they are comparable, by Proposition \ref{canonicalconstruction} applied to the class $\cat_{p,q}$. If $w'$ is a coarsening of $w''$, then the quotient valuation $w''/w'$ is a $p$-henselian valuation on an algebraic extension of $F'$ with residue field $F'$, and hence must be trivial. That is, $w''=w'$. The argument is identical if $w''$ is a coarsening of $w'$.
\end{proof}

\begin{cor}
Suppose $X$ is a smooth, projective variety over $F$, where $F$ is a finite extension of $\mathbb{Q}_p$ containing $\zeta_p$ and $\zeta_q$. Then there is a section of (8) if and only if $X(F) \not = \emptyset$.
\end{cor}
\begin{proof}
Note that the valuation $w'$ of Proposition \ref{bsc} defines an $F'$-rational place of $F(X)$, and hence gives rise to a point in $X(F')$. Indeed, we may always choose a generic point in $F(X)$ with positive value. Its image under the place gives a rational point $a \in X(F')$. Since the restriction map $G_K(p) \rightarrow G_F(p)$ is an isomorphism, $F$ is relatively algebraically closed in $K$, and because $X$ is defined over $F$, in fact $a \in X(F)$, as desired. 
\end{proof}

\begin{cor}
Suppose $X$ is a smooth, projective \emph{curve} over $F$, where $F$ is a finite extension of $\mathbb{Q}_p$ containing $\zeta_p$ and $\zeta_q$. Then every section of (8) lies over a unique $F$-rational point $a \in X(F)$.
\end{cor}
\begin{proof}
This follows from the above corollary at once using Lemma 1.7 from \cite{koe3}. Alternatively, it is a classical result that for curves, all $F$-valuations come from $F$-rational points.
\end{proof}

If we had used maximal solvable quotients instead of maximal $(p,q)$-quotients, we would obtain all the same results, except in this case no extra assumptions need to be made on the presence of roots of unity. In particular:

\begin{cor}
Suppose $X$ is a smooth, projective curve over $F$, where $F$ is a finite extension of $\mathbb{Q}_p$. Then every section of the exact sequence
\begin{equation}
1 \rightarrow G_{\overline{F}(X)}^{solv} \rightarrow G_{F(X)}^{solv} \rightarrow G_F^{solv} \rightarrow 1 \nonumber
\end{equation}
lies over a unique $F$-rational point.
\end{cor}

\bibliographystyle{crelle}
\bibliography{test.bib}

\begin{thebibliography}{10}
\providecommand{\url}[1]{\texttt{#1}}
\providecommand{\urlprefix}{URL }
\expandafter\ifx\csname urlstyle\endcsname\relax
  \providecommand{\doi}[1]{doi:\discretionary{}{}{}#1}\else
  \providecommand{\doi}{doi:\discretionary{}{}{}\begingroup
  \urlstyle{rm}\Url}\fi

\bibitem{ep}
\textit{A.~Engler} and \textit{A.~Prestel}, Valued Fields, Springer-Verlag,
  2005.

\bibitem{jardenritter}
\textit{M.~Jarden} and \textit{J.~Ritter}, On the characterization of the local
  fields by their absolute {G}alois groups, J. Number Th. \textbf{{\bf 11}}
  (1979), 1--13.

\bibitem{jenkner}
\textit{W.~Jenkner}, Les corps p-adiques dont les groupes de {G}alois absolus
  sont isomorphes, Asterisque \textbf{{\bf 209}} (1992), 221--226.

\bibitem{koe1}
\textit{J.~Koenigsmann}, From p-rigid elements to valuations (with a
  {G}alois-characterization of p-adic fields), J. reine angew. {M}ath
  \textbf{{\bf 465}} (1995), 165--182.

\bibitem{koe2}
\textit{J.~Koenigsmann}, Encoding valuations in absolute {G}alois groups,
  Fields {I}nstitute {C}ommunications \textbf{{\bf 33}} (2003), 107--132.

\bibitem{koe3}
\textit{J.~Koenigsmann}, On the section conjecture in anabelian geometry, J.
  reine angew. {M}ath \textbf{{\bf 588}} (2005), 221--235.

\bibitem{strommen_q2}
\textit{J.~Koenigsmann} and \textit{S.~K.}, Recovering valuations on
  {D}emushkin fields, Preprint  (2014).

\bibitem{kpr}
\textit{F.~Kuhlmann}, \textit{M.~Pank} and \textit{P.~Roquette}, Immediate and
  purely wild extensions of valued fields, {M}anuscr. {M}ath. \textbf{{\bf 55}}
  (1986), 39--67.

\bibitem{pop_kenn}
\textit{F.~Pop}, Galoische {K}ennzeichnung p-adisch abgeschlossener
  {K}{\"o}rper, J. reine angew. Math. \textbf{{\bf 392}} (1988), 145--175.

\bibitem{popmeta}
\textit{F.~Pop}, On the birational p-adic section conjecture, Compositio Math.
  \textbf{{\bf 146}} (2010), 621--637.

\bibitem{ritter}
\textit{J.~Ritter}, $p$-adic fields having the same type of algebraic
  extensions, Ann. {M}ath. \textbf{{\bf 238}} (1978), 281--288.

\bibitem{schneps}
\textit{L.~Schneps} and \textit{P.~Lochak}, Geometric {G}alois {A}ctions 1,
  Cambridge, 1997.

\bibitem{serre2}
\textit{J.~Serre}, Galois Cohomology, Springer-Verlag, 2002.

\end{thebibliography}

K.~Str\o mmen, \textsc{Mathematical Institute, Oxford University, Oxford, OX2 6GG}\par\nopagebreak
  \textit{E-mail address}: \texttt{strommen@maths.ox.ac.uk}

\end{document}